\documentclass{amsart}
\usepackage{latexsym}
\usepackage{amsmath}
\usepackage{amssymb}
\usepackage[all]{xy}

\newtheorem{thm}{Theorem}[section]
\newtheorem{prop}[thm]{Proposition}
\newtheorem{cor}[thm]{Corollary}
\newtheorem{lem}[thm]{Lemma}
\newtheorem{claim}[thm]{Claim}
\theoremstyle{definition}
\newtheorem{defn}[thm]{Definition}

\theoremstyle{remark}
\newtheorem{rem}[thm]{Remark}
\newtheorem{ex}[thm]{Example}

\newcommand{\K}{{\mathbb K}}
\newcommand{\Q}{{\mathbb Q}}
\newcommand{\T}{{\mathcal T}}
\newcommand{\e}{\varepsilon}
\newcommand{\D}{\text{D}}
\newcommand{\DD}{\text{\em D}}
\newcommand{\mapright}[1]{%
 \smash{\mathop{%
  \hbox to 1cm{\rightarrowfill}}\limits_{#1}}}
\newcommand{\maprightd}[2]{%
 \smash{\mathop{%
  \hbox to 1.2cm{\rightarrowfill}}\limits^{#1}\limits_{#2}}}
\newcommand{\mapleft}[1]{%
 \smash{\mathop{%
  \hbox to 1cm{\leftarrowfill}}\limits_{#1}}}
\newcommand{\mapleftu}[1]{%
 \smash{\mathop{%
  \hbox to 0.8cm{\leftarrowfill}}\limits^{#1}}}
\newcommand{\maprightu}[1]{%
 \smash{\mathop{%
  \hbox to 1cm{\rightarrowfill}}\limits^{#1}}}
\newcommand{\maprightud}[2]{%
 \smash{\mathop{%
  \hbox to 1cm{\rightarrowfill}}\limits^{#1}_{#2}}}
\newcommand{\mapleftud}[2]{%
 \smash{\mathop{%
  \hbox to 1cm{\leftarrowfill}}\limits^{#1}_{#2}}}

\newcounter{eqn}[section]

\def\theeqn{\textnormal{(\thesection.\arabic{eqn})}}

\def\eqnlabel#1{%
  \refstepcounter{eqn}%
  \label{#1}%
  \leqno{\theeqn}}

\begin{document}

\title{
The ghost length and duality on the chain and cochain type levels}

\footnote[0]{

This research was partially supported by a Grant-in-Aid for Scientific
Research (B) 25287008 
from Japan Society for the Promotion of Science.

Department of Mathematical Sciences, 
Faculty of Science,  
Shinshu University,   
Matsumoto, Nagano 390-8621, Japan   
e-mail:{\tt kuri@math.shinshu-u.ac.jp}
}

\author{Katsuhiko KURIBAYASHI}
\date{}
   
\maketitle

\begin{abstract}
We establish equalities between cochain and chain type levels of maps by making 
use of exact functors which connect appropriate derived and coderived categories.  
Relevant conditions for levels of maps to be finite are extracted from the equalities 
which we call duality on the levels. Moreover, 
we give a lower bound of the cochain type level of the diagonal map on 
the classifying space of a Lie group by considering the ghostness of a shriek map 
which appears in derived string topology. 
A variant of Koszul duality for 
a differential graded algebra is also discussed. 
\end{abstract}

\section{Introduction}
This work is a sequel to previous one \cite{K4, K5}  in which new topological invariants have been studied. 

In \cite{ABIM}, Avramov, Buchweitz, Iyengar and Miller introduced 
a numerical invariant of an objects in a triangulated category, which is called 
the {\it level}. The invariant counts the number of steps to build the given 
object out of some fixed object via triangles. It seems to be {\it cone length} in the category. 
J{\o}rgensen \cite{J, J2, J3} 
developed categorical representation theory of spaces employing the singular (co)chain complexes 
of spaces. In the context of such work, the cochain and 
chain type levels of maps between topological spaces have been defined and studied in \cite{K4, K5}.  

The cochain type level of a map $\alpha : Y \to X$ indeed 
provides a lower bound on the number of 
spherical fibrations which describe a factorization of $\alpha$ in a relevant sense; 
see \cite[Proposition 2.11]{K4}. On the other hand,  
the chain type level of the identity map on a space $Y$ gives an upper bound of the 
L.-S. category of $Y$ in rational case; see \cite[Corollary 2.9]{K5}.  
The L.-S. category is also considered a homotopy invariant counting the number
of cofibrations which construct a given space. 
Therefore, it is natural to anticipate
that the levels of maps inherit duality 
between fibrations and cofibrations, namely Eckmann-Hilton duality. 
For example, one might expect that chain and cochain type levels fit 
into appropriate equalities, which we may call duality on the levels. 
   
In this article, we establish such equalities between 
these two kinds of levels; see Theorem \ref{thm:main} below. 
One of the highlights in getting the result is that we 
make use of a variant of Koszul duality for differential graded
algebras which is given by considering 
exact functors between certain derived and 
coderived categories; see \cite{H-W, Keller2, L-H, P} for Koszul duality. 
In fact, Theorem \ref{thm:diagram} below describes such a variant. 
We can take the singular chains of a space as a coalgebra in Theorem \ref{thm:diagram}. In consequence, 
the commutative diagrams of categories in the theorem give
duality of the levels in Theorem \ref{thm:main}. 
  
Let $X$ be a space and $LX$ the free loop space, namely the space of all continuous maps from the circle $S^1$ to $X$ with 
compact-open topology.   
String topology initiated by the fascinating paper of Chas and Sullivan \cite{C-S} describes a rich structure in 
the homology of 
the free loop space $LM$ of a closed oriented manifold $M$. 
A basic one of the string operations is the so-called loop product on 
the shifted homology $H_{*+\dim M}(LM)$. 
The key to defining these operations is to construct a shriek map (an Umkehr map or a wrong way map) 
associated with the diagonal map on $M$.  

F\'elix and Thomas \cite{F-T} generalized the construction of shriek maps 
on manifolds to that on Gorenstein spaces. 
This enables us to develop string topology in appropriate derived categories; see \cite{K-M-N_1, K-M-N_2} for torsion and extension 
functor descriptions  
of loop (co)products and their applications. 
It is important to mention that the class of Gorenstein spaces contains 
the classifying spaces of connected Lie groups, Borel constructions more general,  
Poincar\'e duality spaces and hence closed oriented manifolds; see \cite{D-G-I, FHT_G, Murillo}.      

Let $BG$ be the classifying space of a connected Lie group $G$. 
In \cite{C-M}, Chataur and Menichi showed that the homology 
$H_*(LBG ;\K)$ with coefficients in a field $\K$ carries the structure of homological conformal field theory 
(HCFT). 
The integration along the fibre of a Borel fibration plays a crucial role in defining the HCFT 
operations. In a derived categorical setting, the integration is considered the homomorphism induced by a 
shriek map on the derived category $\D(C^*(BG\times BG))$ of 
differential graded modules (DG modules) over the cochain 
algebra $C^*(BG\times BG)$ with coefficients in $\K$; see \cite[Theorems 5 and 13]{F-T}. 

Another aim of this article is to consider 
behavior of such shriek maps in the derived category $\D(C^*(BG\times BG))$, more generally in 
$\D(C^*(BG^{\times n}))$.  In particular, we see that non-triviality of a shriek map associated 
with the diagonal map on $BG$ in $\D(C^*(BG^{\times n}))$ gives a lower bound of the {\it ghost length} of $C^*(BG)$; see Theorem \ref{thm:ghost} and Remark \ref{rem:generator}. 
In consequence,  a lower bound of 
the cochain type level of the diagonal map $BG \to BG^{\times n}$ is obtained; 
see Proposition \ref{prop:shriek}.
We mention that the notion of {\it ghosts} has been actually introduced by 
Christensen \cite{Christensen} in a more general framework.   

We conclude this section with comments on topics related to the invariant {\it level}.
Our attempt in \cite{K4, K5} and this paper is closely related to the work in \cite{B-I-K, D-G, R, Shamir}. 
Indeed, the dimension $\dim {\mathcal T}$ of a triangulated category ${\mathcal T}$, which is introduced by Rouquier \cite{R}; 
see also \cite{B-B}, is defined by 
\[
\dim {\mathcal T} = \inf \{d \in {\mathbb N} \ |
 \text{{\tt thick}}^{d+1}_{\T}(C) = {\mathcal T} \ \text{for some object} \ C \ \text{in} \  {\mathcal T} \}.
\]
Here $\text{{\tt thick}}^{j}_{\T}(C)$ denotes the $j$th thickening which is a subcategory of ${\mathcal T}$ 
used when defining the level; see Section 2. 
Thus the dimension gives a global invariant of triangulated categories. 

The results in \cite{B-I-K} due to Benson, Iyengar and Krause are concerned with 
the classification of thick subcategories of a triangulated category. In \cite[2.1 \! Theorem]{D-G},  Dwyer and Greenlees give 
an equivalence between  
categories of torsion and complete modules. Moreover, the result \cite[4.6 Proposition]{D-G} asserts that 
torsion modules are chain complexes built from a fixed complex.  
Then these also clarify global nature of thick or {\it localizing} subcategories.  

On the other hand, the level considered here captures properties of individual objects, 
which come from topological spaces via the singular chain and 
cochain functors; see Remark \ref{rem:level_vs_dimension}  below. 

In \cite{D-G-I}, Dwyer, Greenlees and Iyengar have developed Morita theory in algebraic topology 
by making use of ring spectra. In particular, the result \cite[3.16 Proposition]{D-G-I} describes 
a necessary and sufficient condition for the level of a map to be finite.  
Very recently,  Mao \cite{XFM} has introduced a new numerical invariant for DG modules, 
which is defined by replacing thick subcategories in the definition of the level with localizing ones.  
The invariant of a bounded below DG module coincide with the ghost length plus one; 
see \cite[Theorem A]{XFM}.  

Following \cite{Shamir}, the string topology category invented by 
Blumberg, Cohen and Teleman \cite{B-C-T} can be regarded as a full subcategory of one of the derived categories that we deal with in this paper.   
Then we can expect that machinery used in order to investigate the invariant level is applicable to the study of the string topology category; 
see Remark \ref{rem:string_topology} for such expectation. 
It is worth noting that the recent result \cite[Theorem 1.2]{Shamir} due to Shamir, which is concerned with 
the string topology category,  is deduced by relying on the results in \cite{B-I-K, D-G} cited above.

\section{results} 
 
To describe our results more precisely,   
we first recall from \cite[Section 2]{ABIM} the
definition of the level of an object in a triangulated category $\T$. 
We say that a subcategory of $\T$ is {\it strict} if it is closed under isomorphisms in $\T$. 

For a given object $C$ in $\T$,  
we define the $0$th {\it thickening} 
by $\text{{\tt thick}}^0_{\T}(C)=\{0\}$
and $\text{{\tt thick}}^1_{\T}(C)$ to be the smallest strict full subcategory
which contains $C$ and is closed under taking finite coproducts,
retracts and all shifts. Moreover for $n > 1$ define inductively 
the $n$th thickening $\text{{\tt thick}}^n_{\T}(C)$ 
by the smallest strict full subcategory of 
$\T$ which is closed under retracts and contains objects $M$
admitting a distinguished triangle 
$
M_1 \to M \to M_2 \to \Sigma M_1 
$
in $\T$ for which $M_1$ and $M_2$ are in 
$\text{{\tt thick}}^{n-1}_{\T}(C)$ and $\text{{\tt thick}}^1_{\T}(C)$,
respectively. 
A triangulated subcategory ${\mathcal C}$ of $\T$ is said to be 
{\it thick} if it is closed under taking retracts.   
Then the thickenings provide a filtration
of the smallest thick subcategory 
$\text{{\tt thick}}_{\T}(C)$ of $\T$ containing the
object $C$: 
\[
\{0\} = \text{{\tt thick}}^0_{\T}(C) \subset \cdots \subset 
\text{{\tt thick}}^n_{\T}(C) \subset \cdots \subset 
\cup_{n\geq 0}\text{{\tt thick}}^n_{\T}(C) = \text{{\tt thick}}_{\T}(C). 
\]

For an object $M$ in $\T$, 
we define a numerical invariant $\text{level}_{\T}^C(M)$, which is
called the {\it $C$-level of} $M$,  by 
\[
\text{level}_{\T}^C(M):= \inf \{n \in {\mathbb N} \ |
 M \in  \text{{\tt thick}}^{n}_{\T}(C) \}. 
\]
It turns out that the $C$-level of an object $M$ in $\T$ counts the number
of steps required to build $M$ out of the object $C$ via triangles. 
For more details and general features of the level, we refer
the reader to \cite[Sections 2 and 3]{ABIM}.

Let $\K$ be a field of arbitrary characteristic and $R$ a DG (that is, differential graded) algebra over $\K$. 
Let $\D(R)$ denote the derived category of DG right $R$-modules. Observe that 
the category $\D(R)$ comes equipped with the structure of a triangulated category \cite{Keller}, 
in particular with the shift functor $\Sigma$ defined by 
$(\Sigma M)^{n} = M^{n+1}$. 

We here recall from \cite{K4} and \cite{K5} 
two numerical topological invariants defined by the level 
in a triangulated category $\D(R)$. Unless otherwise explicitly stated, 
it is assumed that a space has the
homotopy type of a connected CW complex whose cohomology with coefficients in the
underlying field is locally finite. 

Let $B$ be a space and $\mathcal{TOP}_B$ the category of maps with the
target $B$; that is, an object of $\mathcal{TOP}_B$ is a map $f : X \to B$ and a morphism form 
$f : X \to B$ to $g : X \to B$ is a map $\alpha : X \to Y$ which satisfies 
the condition that $f= g \circ \alpha$. 
For any object $f : X \to B$,    
the normalized singular cochain $C^*(X; \K)$ with coefficients in $\K$ is regarded as
a DG right module over the cochain algebra $C^*(B; \K)$ 
via the induced map $C^*(f) : C^*(B; \K) \to C^*(X; \K)$. 
Thus the cochain functor gives rise to a contravariant functor 
from the category 
$\mathcal{TOP}_B$ to the triangulated category $\D(C^*(B; \K))$:    
\[
C^*(s(-) ; \K) :  \mathcal{TOP}_B \to \D(C^*(B; \K)), 
\]
where $s(f)$ denotes the source of an object 
$f$ in $\mathcal{TOP}_B$. 

\begin{defn} Let $f$ be an object of $\mathcal{TOP}_B$. 
The {\it cochain type level} of the map $f$ is defined by 
the $C^*(B; \K)$-level of the DG module $C^*(s(f); \K)$, namely 
$\text{level}_{\D(C^*(B; \K))}^{\ C^*(B; \K)}(C^*(s(f); \K))$. 
\end{defn}

Let $F_f$ be the homotopy fibre of a map $f : X \to B$. 
The Moore loop space $\Omega B$ acts on the space $F_f$ 
by the holonomy action. Thus the normalized chain complex 
$C_*(F_f; \K)$ is a DG
module over the chain algebra $C_*(\Omega B; \K)$. 
The normalized singular chain and the homotopy fibre construction enable us 
to obtain a covariant functor  
\[
C_*(F_{(-)} ; \K) :  \mathcal{TOP}_B \to \D(C_*(\Omega B; \K))
\]
from the category $\mathcal{TOP}_B$ 
to the triangulated category $\D(C_*(\Omega B; \K))$.     

\begin{defn} Let $f$ be an object of $\mathcal{TOP}_B$. 
The {\it chain type level} of the map $f$ is defined by the 
$C_*(\Omega B; \K)$-level of the DG module $C_*(F_f; \K)$, namely 
$\text{level}_{\D(C_*(\Omega B; \K))}^{\ C_*(\Omega B; \K)}(C_*(F_f; \K))$. 
\end{defn}

More generally, we call the levels of objects in 
$\D(C_*(\Omega B; \K))$ and in $\D(C^*(B; \K))$ the {\it chain type levels} and
the {\it cochain type levels}, respectively. In what follows, the coefficients
in the singular (co)chain complex and their homology are often omitted if
the context makes them clear.  

\begin{rem}
\label{rem:level_vs_dimension}  Let ${\mathcal T}^c$ be the full subcategory of the triangulated 
category 
${\mathcal T}=\D(A)$ consisting of compact objects, where $A=C^*(S^d; \K)$.   
The result \cite[Proposition 6.6]{S} implies 
that for any $i \in {\mathbb N}$, there exists an indecomposable object $Z_i$ in ${\mathcal T}^c$
such that 
$\text{level}_{\mathcal T}^{A}Z_i =\text{level}_{{\mathcal T}^c}^{A}Z_i = i+1$. 
On the other hand, $\dim {\mathcal T}^c = \infty$.  In fact, if  $\dim {\mathcal T}^c = l < \infty$, 
then there is an object $C$ in ${\mathcal T}^c$ such that 
${\mathcal T}^c =  \text{{\tt thick}}^{l+1}_{\T^c}(C)$. This yields that 
$\text{level}_{\T^c}^{C}M \leq l+1$ for any object $M \in \T^c$.  
Since $C$ is compact, it follows from \cite[Theorem 5.3]{Keller} that 
$\text{level}_{{\mathcal T}}^{A}C=n$ for some $n$. 
Then a triangular inequality (Lemma \ref{lem:thick}) implies that 
\[
\text{level}_{{\mathcal T}}^{A}M \leq 
\text{level}_{{\mathcal T}}^{A}C\cdot \text{level}_{{\mathcal T}}^{C}M \leq n(\dim \T^c +1) 
\]
for any $M$ in $\T^c$. As mentioned above, we have an indecomposable object $Z$ 
in ${\mathcal T}^c$ with $\text{level}_{{\mathcal T}}^{A}Z > n(\dim \T^c +1)$, which is a contradiction. 
\end{rem}

One of our main theorems reveals a remarkable relationship 
between the two kinds of levels. 

\begin{thm}\label{thm:main}
Let $B$ be a simply-connected space and $f : X \to B$ an object in 
$\mathcal{TOP}_B$. Then 
one has (in)equalities 

\medskip
\noindent
$
(1) \ \ \ \ \dim H_*(X; \K) 
\geq \text{\em level}_{\DD(C_*(\Omega B))}^{\ C_*(\Omega B)}(C_*(F_f))
= \text{\em level}_{\DD(C^*(B))}^{\ \K}(C^*(X))$ \ \text{and} \

\medskip
\noindent
$
(2) \ \ \ \ \dim H^*(F_f; \K) 
\geq \text{\em level}_{\DD(C^*(B))}^{\ C^*(B)}(C^*(X))
= \text{\em level}_{\DD(C_*(\Omega B))}^{\ \K}(C_*(F_f)). 
$
\end{thm}

As mentioned in the Introduction, the theorem is deduced from    
a correspondence between the triangulated categories 
$\D(C_*(\Omega B))$ and $\D(C^*(B))$, which
is a variant of Koszul duality for DG algebras; 
see Theorem \ref{thm:diagram}, Proposition \ref{prop:top}  
and Theorem \ref{thm:formal_DGA}. 
More precisely, we deduce 
the results by means of exact functors between 
the triangulated categories which are compatible
with the covariant functor $C_*(F_{(\text{-})})$ and the contravariant functor 
$C^*(s( \text{-}))$. These would allow us to call the equalities in Theorem \ref{thm:main} duality on the (co)chain type levels. 
Algebraic versions of the equalities above deserve mention. They appear in Remark \ref{rem:algebraic_versions}. 

We here describe another evidence that the equalities in Theorem \ref{thm:main},  
which are topological versions, exhibit the duality.  
By definition, the homotopy fibre $F_f$ for a given map $f : X \to B$ fits into a sequence
\[
\xymatrix@C30pt@R15pt{
\Omega B \ar[r]^i & F_f \ar[r]^p & X \ar[r]^f & B
}
\]
in which $p$ is a fibration with $\Omega B$ the fibre. 
We observe that the maps $i$ and $f$ give the chain $C_*(F_f)$ and the cochain $C^*(X)$ a $C_*(\Omega B)$-module structure and 
a $C^*(B)$-module structure, respectively. Since the map $p$ connects those maps $i$ and $f$, 
it seems that (in)equalities in Theorem \ref{thm:main} reflect  homological duality of the fibration in some sense. 
In fact, the Eilenberg-Moore type 
quasi-isomorphism relative to a fibration \cite{FHT, FHT2} 
is an important ingredient for proving the main theorem; 
see Proposition \ref{prop:top}.

Theorem \ref{thm:main} and a {\it triangular inequality} on the
levels (Lemma \ref{lem:thick}) allow us to compare the (co)chain type
levels of maps. 

\begin{prop} \label{prop:L} Under the same assumption as in Theorem
 \ref{thm:main}, one has inequalities 
\begin{eqnarray*}
\text{\em level}_{\DD(C_*(\Omega B))}^{\ C_*(\Omega B)}(C_*(F_f))
&\leq& 
\text{\em level}_{\DD(C_*(\Omega B))}^{\ C_*(\Omega B)}(\K)\cdot
\text{\em level}_{\DD(C^*(B))}^{\ C^*(B)}(C^*(X))\\
&\leq& 
\dim H_*(B; \K)\cdot \text{\em level}_{\DD(C^*(B))}^{\ C^*(B)}(C^*(X))
\  \ \ \ \text{and} \
\end{eqnarray*}
\begin{eqnarray*}
\text{\em level}_{\DD(C^*(B))}^{\ C^*(B)}(C^*(X)) 
&\leq&   
\text{\em level}_{\DD(C^*(B))}^{\ C^*(B)}(\K)\cdot 
\text{\em level}_{\DD(C_*(\Omega B))}^{\ C_*(\Omega B)}(C_*(F_f)) \\
&\leq& 
\dim H^*(\Omega B; \K)\cdot 
\text{\em level}_{\DD(C_*(\Omega B))}^{\ C_*(\Omega B)}(C_*(F_f)). 
\  \ \  \text{}
\end{eqnarray*} 
\end{prop}

As a corollary of Theorem \ref{thm:main}, we have 
criteria for the levels of maps to be finite. 
Let $M$ be an object of the triangulated category $\D(R)$ 
of DG-modules over a DG algebra $R$. 
Then it is immediate that $\dim H(M) < \infty$ if 
$\text{level}_{\D(R)}^{\ \K} M < \infty$.   
Thus we have the following result.

\begin{cor}\label{cor:finite}
Let $f : X \to B$ be a map with $B$ simply-connected. 

\medskip
\noindent
{\em (1)} \  $\text{\em level}_{\DD(C_*(\Omega B))}^{\ C_*(\Omega B)}(C_*(F_f))$
 is finite if and only if so is $\dim H_*(X; \K)$. \\
{\em (2)} \  $\text{\em level}_{\DD(C^*(B))}^{\ C^*(B)}(C^*(X))$
 is finite if and only if so is $\dim H^*(F_f; \K)$.
\end{cor}

This corollary is essentially a special case of \cite[Proposition 2.3]{F-J}; see
also \cite[Theorem 4.8]{ABIM} for other equivalence conditions 
for the level to be finite.   

Let $X$ be a simply-connected rational space. 
The result \cite[Corollary 2.9]{K5} states that  
\[
\text{cat}X \leq 
\text{level}_{\D(C_*(\Omega X))}^{\ C_*(\Omega X)}\Q - 1, 
\] 
where $\text{cat}X$ stands for the L.-S. category of $X$. 
Moreover, a simple calculation in  \cite[Example 6.4]{K5} enables us
to conclude that if $X$ is a
simply-connected rational H-space with $\dim H^*(X; \Q) < \infty$, then 
the above inequality turns out to be the equality.  On the other hand, 
the inequality can be strict as we will see below. 

\begin{ex}
Let $X$ be an infinite wedge of spheres of the form 
$\bigvee_\alpha S^{n_\alpha}$. 
Then $\text{cat}X_\Q = \text{cat} X =1$. By applying Corollary
 \ref{cor:finite} (1) to the case where $id_X : X \to X$, 
we see that 
$\text{level}_{\D(C_*(\Omega X))}^{\ C_*(\Omega X)}\Q = \infty$. In fact,
$H_*(X; \Q)$ is of infinite dimension. 
\end{ex}
  
The following proposition, which is derived from Corollary \ref{cor:finite} (2) 
and the totally fibred square construction \cite[Section 3]{Nei}, is of interest to us. 
Indeed, the result suggests that the study of the levels for maps contributes to 
determining the homotopy types of spaces. 

\begin{prop} \label{prop:section} Let $\pi : X \to B$ be a map between
 simply-connected spaces with a right homotopy inverse $s : B \to X$. 
 We regard the map $\pi$ and $s$ as objects in $\mathcal{TOP}_B$ and 
 $\mathcal{TOP}_X$, respectively.  Then both of levels 
$\text{\em level}_{\DD(C^*(B; \K))}^{\ C^*(B; \K)}C^*(X; \K)$ and 
$\text{\em level}_{\DD(C^*(X; \K))}^{\ C^*(X; \K)}C^*(B; \K)$ are finite 
if and only if $H^*(F_\pi; \K)=\K$. In particular, the both of the two
 levels with coefficients in  ${\mathbb Z}/p$ are finite if and only if 
 $\pi : X \to B$ is a homotopy equivalence after $p$-completion. 
\end{prop}

Let $q : X \to B$ be a trivial fibration with $H^*(F_q; \K)$ finite
dimensional. Then the
$C^*(B;\K)$-level of $C^*(X;\K)$ in $\D(C^*(B;\K))$ is just one for any
field $\K$ in general. In fact, we see that 
$C^*(X; \K)\cong C^*(B; \K)\otimes H^*(F_q; \K)$ in 
$\D(C^*(B; \K))$. 
This implies that $C^*(X; \K)$ is a coproduct of shifts of $C^*(B; \K)$
and hence $C^*(X; \K)$ is in 
${\tt thick}^1_{D(C^*(B; \K))}(C^*(B; \K))$. 
On the other hand, 
for a spherical fibration $S^l \to X \to B$, 
we obtain a characterization for the $C^*(B;\K)$-level of $C^*(X;\K)$ to
be two; see Proposition \ref{prop:level_2}. Combining the result 
with Proposition \ref{prop:L}, we have the following proposition.  

\begin{prop}\label{prop:level_2^n}
Let $B$ be a simply-connected space. Suppose that there exists a
 sequence of fibrations
\[
F_1 \longrightarrow X_1 \stackrel{p_1}{\longrightarrow} B, \ F_2 \longrightarrow X_2 \stackrel{p_2}{\longrightarrow} X_1,
 \ ..., \  F_n \longrightarrow X_n \stackrel{p_n}{\longrightarrow} X_{n-1}
\]
in which $X_i$ is simply-connected for $1 \leq i < n$ and 
$H^*(F_i; \K)\cong H^*(S^{n_i}; \K)$ for some $n_i$. 
Then one has inequalities 
\[
\text{\em level}_{\DD(C^*(B; \K))}^{\ C^*(B; \K)}C^*(X_n; \K)\leq 2^n \
 \ 
\text{and}   
\]
\[ 
\text{\em level}_{\DD(C_*(\Omega B; \K))}^{\ C_*(\Omega B; \K)}
C_*(F_f; \K)\leq 2^n\cdot  \dim H_*(B; \K),
\]
where $f = p_n \circ \cdots \circ p_1$. 
\end{prop}

In rational case, the result \cite[Proposition 2.7]{K4} gives a
better estimate of $C^*(B; \Q)$-level of $C^*(X_n; \Q)$ 
than that of Proposition \ref{prop:level_2^n} provide each 
$n_i$ is odd. 

In order to describe another main theorem,  
we recall a numerical invariant for DG modules related to the level. Let $A$ be a DG algebra.  
We call a morphism $f: M \to N$ 
in the derived category $\D(A)$ 
a {\it ghost} if $H(f)=0$. 
An object $M$ in $\D(A)$ is said to have
{\it ghost length} $n$, denoted $\text{gh.len.} M = n$, if every composite 
\[
\xymatrix@C25pt@R25pt{
M \ar[r]^{f_1}&  Y_1  \ar[r]^{f_2} & \cdots 
\ar[r]^{f_{n+1}} & Y_{n+1}
}
\]
of $n+1$ ghosts is trivial in $\D(A)$, and there exists a composite of
$n$ ghosts from $M$ which is non trivial in $\D(A)$; see \cite{Ho-L}.

The ghost length of a DG module $M$  gives a lower bound of the level of $M$. 

\begin{prop} \cite[Lemma 6.7]{S} \cite[Proposition 7.5]{K5}
\label{prop:ghost}
For any $M \in \DD(A)$, one has 
\[
\text{\em gh.len.} M +1 \leq \text{\em level}_{\DD(A)}^A(M).  
\]
\end{prop}

Let $BG$ be the classifying space of a connected Lie group $G$.  Since the diagonal map $\Delta : G \to G \times G$ is a homomorphism, 
it induces a map $BG \to BG^{\times 2}$, which is regarded as the diagonal map $BG \to BG\times BG$ 
under a homotopy equivalence between $BG^{\times 2}$ and $BG \times BG$. 
We give an estimate for the cochain type level of the composite 
\[
\xymatrix@C25pt@R25pt{
\Delta^{(n-1)} : 
BG \ar[r]^(0.6){B\Delta} & BG^{\times 2}  \ar[r] & \cdots \ar[r]^{B(1\times \Delta)} & BG^{\times n} 
}
\]
by considering the ghostness of a shriek map associated with 
the map $B(1\times\Delta) : BG^{\times l} \to BG^{\times (l+1)}$; see \cite{F-T} and Section 5 for shriek maps on  
a Gorenstein space.

\begin{thm} \label{thm:ghost} Let $BG$ be the classifying space of a connected Lie group $G$ 
whose cohomology with coefficients in $\K$ 
is isomorphic to a polynomial algebra. Then in the derived category 
$\DD(C^*(BG^{\times n}))$, one has 
$$n-1 \leq \text{\em gh.len.} C^*(BG).$$ 
\end{thm}

The assumption for a Lie group $G$ in Theorem \ref{thm:ghost} is satisfied for any field  $\K$ if the homology 
$H_*(G; {\mathbb Z})$ is torsion free. 
Moreover, the classical Lie groups $SO(n)$, $Spin(n)$ for $n\leq 9$, the exceptional Lie groups $G_2$ and $F_4$ 
satisfy the assumption in the case where the field $\K$ is of characteristic $2$ 
while the integral homology groups of these Lie groups have $2$-torsion; see \cite{M-T}.  

The proof of Theorem \ref{thm:ghost} uses the Leray-Serre and the Eilenberg-Moore spectral sequences. 
The key to the proof is the non-triviality of the loop coproduct in string topology on the classifying space of a Lie group  \cite{C-M, K-M_BG}. 
Therefore it is hard to expect an algebraic proof of the theorem. 

By Proposition \ref{prop:ghost} and Theorem  \ref{thm:ghost}, we have the following result.  

\begin{prop}\label{prop:shriek}
Under the same assumption as in Theorem \ref{thm:ghost}, 
\[
n \leq \text{\em level}_{\DD(C^*(BG^{\times n})}^{C^*(BG^{\times n})}(C^*(BG)) \leq (n-1) \dim QH^*(BG; \K)+1, 
\]
where $QH^*(BG; \K)$ stands for the vector space of indecomposable elements of the algebra $H^*(BG; \K)$. 
Assume further that $QH^*(BG; \K)^{2j+1}=0$ for $j\geq 0$. Then 
\[
\text{\em level}_{\DD(C^*(BG^{\times n}))}^{C^*(BG^{\times n})}(C^*(BG)) = 
 (n-1) \dim QH^*(BG; \K)+1. 
\]
\end{prop}

For example, we consider the orthogonal group $SO(3)$. Since the mod $2$ cohomology 
$H^*(BSO(3); {\mathbb Z}/2)$ is a polynomial algebra generated by the second and the third 
Stiefel-Whitney classes, it follows that 
\[
2 \leq {\text{gh.len.}}C^*(BSO(3);{\mathbb Z}/2) +1 \leq 
\text{\rm level}_{\D(C^*(BSO(3)^{\times 2}))}^{C^*(BSO(3)^{\times 2})}(C^*(BSO(3)) 
\leq 3.
\]

\begin{rem} 
Let $F_{\Delta^{(n-1)}}$ be the homotopy fibre of the map $\Delta^{(n-1)} : BG \to BG^{\times n}$. 
Then the fibration $F_{\Delta^{(n-1)}} \to BG$ admits the holonomy right action of $\Omega BG^{\times n}$ 
and is weakly equivalent to the fibration $BG\times_{BG^{\times n}}EG^{\times n} \to BG$ with 
the holonomy right action of $G^{\times n}$; see \cite[Proposition 2.11]{F-H-T} for example.  Then 
the duality in Theorem \ref{thm:main} (2) implies that 
\begin{eqnarray*}
\text{level}_{\D(C^*(BG^{\times n}))}^{C^*(BG^{\times n})}(C^*(BG))&=& 
\text{level}_{\D(C_*(G^{\times n}))}^{\K}(C_*(BG\times_{BG^{\times n}}EG^{\times n})).  
\end{eqnarray*}
We observe that $BG\times_{BG^{\times n}}EG^{\times n}$ is homotopy equivalent to 
a homogeneous space of the form 
$G^{\times n}/\Delta G = G^{\times (n-1)}$, where 
$\Delta : G \to G^{\times n}$ denotes the diagonal map. In fact, we have a homotopy fibre square 
\[
\xymatrix@C25pt@R20pt{
G^{\times n}/\Delta G \ar[r]^-\simeq \ar@{=}[d]& 
BG\times_{BG^{\times n}}EG^{\times n} \ar[r]^-{} \ar[d] & EG^{\times n}
\ar[d]  \\
G^{\times n}/\Delta G \ar[r]  \ar[r]  &BG \ar[r]_{\Delta} & BG^{\times n}.
}
\]
\end{rem}

The rest of the article is organized as follows. 
In Section 3, we prove Theorem \ref{thm:main}. Section 4 is devoted to
proving Propositions  \ref{prop:L}, \ref{prop:section} and \ref{prop:level_2^n}. 
Section 5 presents proofs of Theorem \ref{thm:ghost} and Proposition \ref{prop:shriek}. 
In Appendix, we recall results on 
a coderived category due to Lef\`evre-Hasegawa \cite{L-H} on which we
rely when proving Theorem \ref{thm:main}. 
Moreover, a variant of Koszul duality due to He and Wu \cite{H-W} is discussed.

\section{Proof of Theorem \ref{thm:main}} 

As we will see below, for an object $f$ in ${\mathcal TOP}_B$, we obtain  
quasi-isomorphisms which connect DG modules $C^*(s(f))$ and
$C_*(F_f)$  by making use of the bar and cobar constructions. 
In order to prove Theorem \ref{thm:main}, we incorporate 
such the quasi-isomorphisms into arguments on
appropriate derived and coderived categories.   

For a graded vector space $V$, we denote by $V^\vee$ 
the graded dual $\text{Hom}_\K(V, \K)$, namely 
$(V^\vee)^{-k}=(V^\vee)_k = \text{Hom}_\K(V^k, \K)$.  We say that $V$ is {\it locally finite} if $V^i$ is of finite dimension for each $i$. 

\begin{defn}
(i) Let $A$ be an augmented DG algebra over $\K$ 
with differential of degree $+1$ and $\mathsf{mod-}A$ the category of DG left $A$-modules. 
The {\it derived category} $\D(\mathsf{mod-}A)$ of 
DG right $A$-modules is the localization of the homotopy category of $\mathsf{mod-}A$ 
at the class of quasi-isomorphisms. \\
(ii) Let $C$ be a co-augmented DG coalgebra over $\K$ with differential of degree $+1$ and $\mathsf{comod-}C$ the 
category of cocomplete DG right $C$-comodules; see the Appendix.
The {\it coderived category} $\D(\mathsf{ comod-}C)$ of 
cocomplete DG $C$-comodules is the localization of the homotopy category of $\mathsf{comod-}C$ 
at the class of weak equivalences; 
see \cite{L-H} and also the Appendix.

We shall write $\D(A)$ and $\D_c(C)$ for $\D(\mathsf{mod-}A)$ and 
$\D(\mathsf{comod-}C)$, respectively.
\end{defn}

By definition, a {\it simply-connected} algebra $A$ satisfies the condition that $A^0 =\K$, $A^1=0$ and $A^i=0$ for $i < 0$.   
We call a coalgebra $C$ {\it simply-connected} if $C^0=\K$, $C^{-1} =0$ and $C^i=0$ for $i > 0$. 
In what follows, we assume that an algebra and a coalgebra are endowed 
with an augmentation and a co-augmentation, respectively and that they are defined over a field $\K$. 

Let $F : \mathsf{comod-}C \to C^\vee\mathsf{-mod}$ be a functor given by
sending a cocomplete DG right $C$-comodule to the DG left $C^\vee$-module with the same
underlying $\K$-module and 
whose multiplication is given by the natural composite
\[
C^\vee \otimes M \to C^\vee \otimes M \otimes C \to C^\vee \otimes C
\otimes M \to M. 
\]
Composing the vector space dual functor $( \ )^\vee$ with $F$, we have an
exact functor
\[
tD : \D(\mathsf{comod-}C) \ \maprightu{F_*} \ \D(C^\vee \mathsf{-mod}) \ 
\maprightu{( \ )^\vee} \ \D(\mathsf{mod-}C^\vee)
\]
from the coderived category to the derived category; see Remark \ref{rem:coalgebra}.

We deal with the bar and cobar constructions below. For the (co)algebra
and (co)module structures of these constructions, 
see \cite[Section 2]{FHT}, \cite[Section 4]{FHT2} and \cite{Mu}. We also refer the reader to \cite{H-M-S} for 
differential graded objects.

Let $A$ be a DG algebra and 
consider the bar resolution ${B}(A; A) \to {}_A\K$ of $\K$. 
Let ${B}(A)$ be a DG coalgebra defined by 
${B}(A)=\K\otimes_A{B}(A; A)$. 
By using the twisted tensor product construction associated with 
the natural twisting cochain $\tau : {B}(A) \to A$ of degree $+1$, 
we have a pair of adjoint functors 
\[
\xymatrix@C30pt@R25pt{
\D_c({B}(A)) \ar@<1ex>[rr]^{L:= - \otimes_{\tau}A} & & 
\D(A)   \ar@<1ex>[ll]^{R:= - \otimes_{\tau}{B}(A)}. 
}
\]
For more details, see \cite{Mu}, \cite[Ch. 2]{L-H}, 
\cite{Keller2} and also Appendix.  
We write $R_A$ for the functor $- \otimes_\tau{B}(A)$. 
The definition of the twisted tensor product enables us to deduce that 
$R_A$ coincides with the functor $- \otimes_A{B}(A; A)$. For a right $A$-module $M$, we may write $B(M; A)$ for $M\otimes_AB(A; A)$. 
For a coalgebra $C$ and a right $C$-comodule $N$, let $\Omega (N ; C)$ denote the cobar construction; 
see \cite[Section 2]{FHT} for example.  
  
The duality on the bar and cobar constructions yields the following result. 

\begin{prop}\label{prop:equivalence} Let $C$ be a simply-connected DG coalgebra with $H(C)$ locally finite. Then 
there exists an equivalence
\[
\Theta : \DD(\Omega C) \to \DD({B}(C^\vee)^\vee) 
\]
of triangulated categories such that for a DG $C$-comodule $N$ with $H(N)$ locally finite and 
bounded above,   
\[
\Theta(\Omega (N; C))\cong {B}(N^\vee; C^\vee)^\vee.
\]
\end{prop}

\begin{proof}
Let $u : A =TV \stackrel{\simeq}{\longrightarrow} C^\vee$ be a TV-model for the simply-connected DG algebra $C^\vee$ in 
the sense of Halperin and Lemaire \cite{H-L}. By assumption, 
$H(C^\vee)$ is locally finite. Then without loss of generality, 
we can assume that $A$ is also locally finite; see \cite[Proposition 4.2]{FHT}.
 
Let $\Delta : C \to C \otimes C$ be the comultiplication
 on $C$. Then the multiplication 
$m : C^\vee \otimes C^\vee \to C^\vee$ is defined by the composite 
\[
C^\vee \otimes C^\vee=C^\vee \otimes C^\vee \stackrel{q'}{\to}
(C \otimes C)^\vee \stackrel{\Delta^\vee}{\to} C^\vee =
 C^\vee, 
\]
where $q'$ denotes the natural quasi-isomorphism. 
We have a commutative diagram 
\[
\xymatrix@C37pt@R20pt{
A^\vee \ar[r]^{m^\vee} & (A\otimes A)^\vee & A^\vee\otimes A^\vee
 \ar[l]^{\cong}_{q''} \\ 
(C^\vee)^\vee \ar[r]^(0.45){m^\vee} \ar[u]^{u^\vee}& (C^\vee\otimes C^\vee)^\vee \ar[u]_{(u\otimes u)^\vee}&
 C^{\vee\vee}\otimes C^{\vee\vee}  \ar[l]_{q'}^{\simeq}  \ar[u]_{u^\vee \otimes u^\vee} \\ 
C \ar[u]^{q}_{\simeq} \ar[rr]_{\Delta} & & C \otimes C,  
\ar[u]_{q\otimes q}
}
\]
where $q$ and $q''$ are the natural quasi-isomorphisms. 
In fact, the commutativity of the lower square follows from 
that of the diagram 
\[
\xymatrix@C37pt@R20pt{
(C^\vee)^\vee \ar[r]^(0.45){\Delta^{\vee \vee}} &
(C \otimes C)^{\vee \vee} \ar[r]^{(q')^\vee}& (C^\vee\otimes
 C^\vee)^\vee \\
C \ar[u]^{q} \ar[r]_(0.45){\Delta} &  C \otimes C \ar[u]_{q}
\ar[r]_{q\otimes q} & C^{\vee\vee} \otimes C^{\vee\vee}. \ar[u]_{q'} 
}
\]
Observe that $q'' : A^\vee \otimes A^\vee \to (A\otimes A)^\vee$ 
is an isomorphism because $A$ is locally finite. This implies that 
$u^\vee \circ q : C \to A^\vee$ is a quasi-isomorphism of 
coalgebras. We then have a sequence of quasi-isomorphisms of algebras
\[
\xymatrix@C37pt@R20pt{
\Omega C \ar[r]^{\rho:=\Omega(u^\vee q)}_{\simeq} & 
\Omega (A^\vee)  \ar[r]_{\cong}^{\mu_1} & {B}(A)^\vee  &
{B}(C^\vee)^\vee 
\ar[l]_{\mu_2}^{\simeq}. 
}
\eqnlabel{add-3}
\]
Thus the result \cite[Proposition 4.2]{K-M} 
enables us to obtain equivalences of triangulated categories
\[
\xymatrix@C37pt@R12pt{
\D(\Omega C) \ar@<1ex>[r]^{- \otimes^L_{\Omega C}\Omega (A^\vee)} 
& \D(\Omega (A^\vee))  \ar@<1ex>[l]^{\rho^*}_{\simeq} 
\ar@<1ex>[r]^{- \otimes^L_{\Omega (A^\vee)}{B}(A)^\vee}
& \D({B}(A)^\vee)  \ar@<1ex>[l]^{\mu_1^*}_{\simeq} 
\ar[r]^(0.45){\mu_2^*}_(0.45){\simeq}& 
\D({B}(C^\vee)^\vee).  
}
\]
We define $\Theta : \D(\Omega C) \to \D({B}(C^\vee)^\vee)$ by
the composite. 

Let $C_1 \stackrel{\simeq}{\longrightarrow} N^\vee$ 
be an $A$-semifree resolution for $N^\vee$; see
\cite[Propositions 4.6 and 4.7]{FHT}. 
Since $H_*(N)$ is locally finite, 
we may assume that so is the $A$-module $C_1$; see \cite[Proposition 4.6]{FHT}. 
Then we can define a comodule structure on $C_1^\vee$ by the composite 
$c : C_1^\vee \longrightarrow (C_1\otimes A)^\vee \stackrel{\cong}{\longleftarrow}  
C_1^\vee \otimes A^\vee 
$. 
The same argument as above allows us to obtain a commutative diagram
\[
\xymatrix@C40pt@R15pt{
C_1^\vee \ar[r]^-{c} & C_1^\vee \otimes A^\vee \\
N \ar[r]_-{\Delta_N} \ar[u]^{\simeq} 
& N \otimes C \ar[u]_{\simeq}
}
\]
in which vertical arrows are quasi-isomorphisms. 
Thus we have an isomorphism 
$\Omega (N; C) \cong \rho^*\Omega (C_1^\vee; A^\vee)$ 
in $\D(\Omega C)$. Moreover, it follows from the locally finiteness
 of $C_1$ and $A$ that
$\Omega (C_1^\vee; A^\vee)$ is isomorphic to 
$\mu_1^*({B}(C_1; A)^\vee)$ and 
$\mu_2^*({B}(C_1; A)^\vee)\cong {B}(N^\vee; C^\vee)^\vee$ in 
$\D({B}(C^\vee)^\vee)$. 
This completes the proof.  
\end{proof}

We have a crucial result on exact functors which connect the triangulated
categories $\D(\Omega C)$ and $\D(C^\vee)$ for a coalgbera $C$.  
The result is a key to proving the duality on chain and cochain type levels
described in Theorem \ref{thm:main}. 

Let $A$ and $C$ be an augmented DG algebra and a co-augmented cocomplete DG coalgebra, respectively. 
The result \cite[Proposition 2.14]{FHT} asserts that for a $C$-comodule $N$,  
there exist a quasi-isomorphism
$\sigma_C : C \stackrel{\simeq}{\to} B\Omega C$ of coalgberas and  a quasi-isomorphism 
$\sigma_N : N \stackrel{\simeq}{\to} B(\Omega (N; C) ; \Omega C)$ of $C$-comodules.  
 
\begin{thm}
\label{thm:diagram}  {\em (i)} Under the same assumption as above on the coalgebra $C$, 
one has a commutative diagrams up to isomorphism 
\[
\xymatrix@C40pt@R12pt{
 & & \DD_c(C) \ar@/_/[lld]_{\Omega( \ ; C)} 
\ar@/^/[rd]^-{tD}   & \\
\DD(\Omega C) \ar[r]_-{R_{\Omega C}} &  \DD_c({B}\Omega C) 
\ar[r]_-{tD}
 & \DD(({B}\Omega C)^\vee) 
\ar@<-1ex>[r]_-{- \otimes^L_{(B\Omega C)^\vee}C^\vee}^-{\simeq} & \DD(C^\vee); 
\ar@<-1ex>[l]_-{(\sigma_C^\vee)^*} 
}
\]
that is, there exists a natural isomorphism between two composite functors from $\DD_c(C)$ to $\DD(({B}\Omega C)^\vee)$. 
Moreover, all the functors between (co)derived categories are exact.  

\noindent
{\em (ii)} Let $C$ be a simply-connected DG coalgbera with $H(C)$ locally finite. Let $\DD^{lf, -}_c(C)$ denote the full subcategory of $\DD_c(C)$ 
consisting of comodules whose cohomologies are locally finite and bounded above.  Then one has a commutative diagram up to isomorphism 
\[
\xymatrix@C40pt@R10pt{
 & & \DD^{lf,-}_c(C) \ar@/_/[lld]_{\Omega( \ ; C)} 
\ar@/^/[rd]^{tD}   & \\
\DD(\Omega C) \ar[r]_-{\Theta}^-{\simeq} &  \DD((B(C^\vee))^\vee) 
 & \DD_c(B(C^\vee)) \ar[l]^-{tD} & \DD(C^\vee) \ar[l]^-{R_{C^\vee}}   
}
\]
in which all the functors are exact.  
\end{thm}

\begin{proof}
(i) Let $\tau : C \to \Omega C$ be the canonical twisting cochain. Then 
$\Omega ( \ ; C)$ is nothing but the functor $L= \text{-} \otimes_{\tau}\Omega C$ mentioned 
in Theorem \ref{thm:L-R} below. 
In particular, we see that $\Omega( \ ; C)$ is exact. 
Moreover, it follows that for any $N$ in $\D_c(C)$, 
\[
tD(B(\Omega (N; C); \Omega C)) = B(\Omega (N; C); \Omega C)^\vee \simeq N^\vee
= (\sigma_C^\vee)^*tD(N)
\]
in $\text{D}((B\Omega C)^\vee)$. 
This implies that the diagram is commutative up to isomorphism. 

(ii) Proposition \ref{prop:equivalence} yields that for any $N$ in  $\text{D}^{lf, -}_c(C)$,  
\[
\Theta(\Omega(N ; C)) \cong B(N^\vee ; C^\vee)^\vee = (tD\circ R_{C^\vee}\circ tD)(N). 
\]
We have the result. 
\end{proof}

In order to prove Theorem \ref{thm:main}, we recall important results on the level.  

\begin{lem}\label{lem:1}\cite[Lemma 3.9]{S}
Let $M$ be a DG-module over a non-negative simply-connected or non-positive connected DG algebra $A$. 
Assume that $M$ is bounded below if $A$ is non-negative and is 
bounded above if $A$ is non-positive. Then 
\[
\dim H(M \otimes_A^L\K) \geq \text{\em level}_{\DD(A)}^A(M).
\]
\end{lem}

The difference between the dimension of $H(M \otimes_A^L\K)$ and the
level is also of interest to us. In general, the difference is very large.  
The proof of Lemma \ref{lem:1} which we provide below exhibits the fact. 

\medskip
\noindent
{\it Proof of Lemma \ref{lem:1}.} 
If $\dim H(M \otimes_A^L\K) =\infty$, then  the assertion is immediate. 
 
By assumption, the module $M$ admits a minimal semi-free resolution  
$F \stackrel{\simeq}{\to} M$ endowed with 
a filtration $\{F^l\}_{l\geq 0}$ of $F$; see
\cite{FHT_G}, \cite{F-H-T} and \cite[Section 2]{FHT2} . 
We thus obtain triangles 
$\coprod_i\Sigma^{n_i^0}A \to F^1 \to \coprod_j\Sigma^{n_j^1}A \to ,$   
$F^1 \to F^2 \to \coprod_j\Sigma^{n_j^2}A \to  , ... ,$   
$F^{n-1} \to F^n \to \coprod_j\Sigma^{n_j^n}A \to , ... . $  
The minimality of the semi-free resolution enables us to deduce that 
\[
H(M\otimes_A^L\K)=H(F\otimes_A \K)=F\otimes_A \K 
= \coprod_{s\geq 0}\coprod_j\Sigma^{n_j^s}\K. 
\]
Suppose that $\dim H(M \otimes_A^L\K)$ is of finite dimension.   
Then it follows that there 
exists an integer $n$ such that $F^n \simeq M$ and each index $j$
runs in finite numbers. 
Thus we see that $M \in \text{{\tt thick}}^{n+1}_{\D(A)}(A)$. 
If $n_j^s=0$ for any $j$ and $s$, then $M \simeq 0$ and hence the result is
obvious. Without loss of generality, we can assume 
that for any $s$, $\coprod_j\Sigma^{n_j^s}A$ is non-trivial. 
We then have 
\[
\dim H(M \otimes_A^L\K) = \dim \coprod_{s=0}^n\coprod_j\Sigma^{n_j^s}\K \geq n+1 
\geq \text{level}_{\D(A)}^{\ A}(M). 
\]
This completes the proof. 
\qed\hfill

\medskip
Let $\gamma : {\mathcal T} \to {\mathcal U}$ be an exact functor of 
triangulated categories. Then we have the following result.

\begin{lem}\label{lem:2}\cite[Theorem 2.4 (6)]{ABIM} 
$
\text{\em level}_{{\mathcal T}}^C(M) \geq \text{\em level}_{{\mathcal U}}^{\gamma(C)}(\gamma(M)).
$
\end{lem}

Let $B$ be a simply-connected space and $f : X \to B$ a map. 
Recall from \cite[Theorem II]{FHT} a quasi-isomorphism of DG-modules
\[
\Phi : \Omega(C_*(X) ; C_*(B)) \ \maprightud{\simeq}{} \  C_*(F_f),   
\eqnlabel{add-0}
\]
which is compatible with actions of $\Omega C_*(B)$ and $C_*(\Omega B)$ via 
a quasi-isomorphism of DG algebras
\[
\phi : \Omega C_*(B) \to C_*(\Omega B),
\eqnlabel{add-1}
\]
where $\Omega(N ; C_*(B))$ denotes the cobar construction of
the right $C_*(B)$-comodule $N$. 
We mention that the quasi-isomorphisms $\Phi$ and $\phi$ are induced from
the {\it universal constructions} due to Adams \cite{A}; 
see also \cite[Section 3]{FHT}.

We connect the category $\mathcal{TOP}_B$ with $\D(C_*(\Omega B))$ and $\D(\Omega C_*(B))$. 

\begin{prop}\label{prop:top} Let $B$ be a simply-connected space. One has a commutative diagram 
up to isomorphism 
\[
\xymatrix@C40pt@R25pt{
\DD(C_*(\Omega B))\ar@<1ex>[d]^{\phi^*}& \mathcal{TOP}_B \ar[l]_{C_*(F_{( - )})} 
\ar@/^/[rd]^{C^*(s( \ ))}   \ar[d]_{C_*(s( \ ))}& \\
\DD(\Omega C_*(B)) \ar@<1.5ex>[u]^(0.6){- \otimes^L_{\Omega C_*(B)}C_*(\Omega B) \ \ \ }_{\simeq} &  
\DD^{lf, -}_c(C_*(B)) \ar[l]^{\Omega ( \ ; C_*(B))} \ar[r]_{tD}
 & \DD(C^*(B)).  
}
\]
\end{prop}

\begin{proof}
The quasi-isomorphisms in (3.2) and (3.3) enable us to conclude that the left hand-side square is commutative up to isomorphism.  
By definition, the right hand-side triangle is commutative. 
\end{proof}

\begin{rem}\label{rem:1}
We write $\eta$ and $\nu$ for 
the composites 
$\text{-}\otimes^L_{(B\Omega C)^\vee}C^\vee \circ tD\circ R_{\Omega C_*(B)}\circ \phi^*$ 
and 
$\text{-} \otimes^L_{\Omega C_*(B)}C_*(\Omega B)\circ \Theta^{-1}\circ tD \circ R_{C^*(B)}$, respectively; 
see Theorem \ref{thm:diagram} and 
Proposition \ref{prop:top}. 
It follows that $\eta (C_*(\Omega B)) \cong \eta(C_*(F_{(\ast \to B)})) \cong C^*(\ast) \cong \K_{C^*(B)}$ and 
$\nu(C^*(B)) \cong \nu C^*(s(id : B \to B)) \cong C_*(F_{(id)}) \cong C_*(\ast) \cong \K_{C_*(\Omega B)}$. Moreover, we see that 
$\eta(\K) \cong \eta(C_*(F_{(id :B \to B)})) \cong C^*(B)$ and $\nu(\K) \cong \nu(C^*(s(\ast \to B))) \cong C_*(F_{(\ast \to B)}) 
\cong C_*(\Omega B)$. 
\end{rem}

We are now ready to prove our main theorem. 

\medskip
\noindent
{\it Proof of Theorem \ref{thm:main}.}
It follows from \cite[Proposition 19.2]{F-H-T} that 
${B}(C_*(\Omega B); C_*(\Omega B))$ is a 
$C_*(\Omega B)$-semifree resolution of $\K$. Then the result 
\cite[Proposition 6.7]{FHT2} yields that 
\[
C_*(X)\simeq C_*(F_f)\otimes_{C_*(\Omega B)}
{B}(C_*(\Omega B); C_*(\Omega B)) \simeq
C_*(F_f)\otimes^L_{C_*(\Omega B)}\K.
\]
By virtue of Lemma \ref{lem:1}, 
we have the first inequality.  
 
Let $F_f \to E_X \to B$ be the fibration associated with the map 
$f : X \to B$. Since there exists a homotopy equivalence $j : X \to E_X$
which is in ${\mathcal TOP}_B$, it follows that, as vector spaces,  
\[
H^*(C^*(X)\otimes^L_{C^*(B)}\K)\cong \text{Tor}_{C^*(B)}(C^*(X), \K)
\cong \text{Tor}_{C^*(B)}(C^*(E_X), \K) \cong H^*(F_f). 
\]
Observe that the third isomorphism is induced by the Eilenberg-Moore
map; see for example \cite[Theorem 3.3]{G-M}. 
By applying Lemma \ref{lem:1} again,  
one has the second inequality.   

It follows from Lemma \ref{lem:2}, Theorem \ref{thm:diagram} (i), Proposition \ref{prop:top}  
and Remark \ref{rem:1} that 
\[
\text{level}_{\D(C_*(\Omega B))}^{\ C_*(\Omega B)}(C_*(F_f))
\geq \text{level}_{\D(C^*(B))}^{\ \K}(C^*(X)).
\]
Theorem \ref{thm:diagram} (ii) yields 
the converse inequality. 
The same argument as above works well to obtain the equality in (2). 
\hfill \qed

\begin{rem}\label{rem:algebraic_versions}
Let $C$ be a co-augmented cocomplete $DG$ coalgebra with $H(C)$ locally finite 
and $M$ an object in $\D^{lf, -}(C)$. Then Theorem \ref{thm:diagram} yields 
algebraic versions of equalities in Theorem \ref{thm:main}. Indeed, we have equalities 
\[
\text{level}_{\D(\Omega C)}^{\Omega C}(\Omega(M; C)) = \text{level}_{\D(C^\vee)}^\K(M^\vee) \ \ \text{and}  
\]
\[
\text{level}_{\D(C^\vee)}^{C^\vee}(M^\vee) = \text{level}_{\D(\Omega C)}^\K(\Omega(M; C)). 
\]
\end{rem}

\begin{rem}\label{rem:string_topology}
As mentioned in the Introduction, the string topology category $\mathsf{St}_{M}$ for 
a simply-connected oriented manifold $M$ is a full subcategory of $\D(C_*(\Omega M))$; see \cite{Shamir}. 
Then Proposition \ref{prop:top} and Theorem \ref{thm:diagram} may generalize the result \cite[Theorem 2.8]{B-C-T} 
on the Dwyer-Kan equivalence between $\mathsf{St}_{M}$ and the full subcategory of $\D(C^*(M))$-modules consisting 
of objects in the image of the functor $C^*(s( \  ))$.  
This will be discussed in a forthcoming paper \cite{K-M_Loop}. 
\end{rem}

\section{Proofs of Propositions \ref{prop:L}, \ref{prop:section} and  
\ref{prop:level_2^n}}

We here recall some full subcategories of a triangulated category 
${\mathcal T}$ before proving Proposition \ref{prop:L}.

Let ${\mathcal A}$ be a subcategory of ${\mathcal T}$ and 
${\tt add}^{\Sigma}({\mathcal A})$ the smallest full subcategory of 
${\mathcal T}$ that contains ${\mathcal A}$ and is closed under finite
coproducts, all shifts and isomorphisms. 
The category ${\tt smd}({\mathcal A})$ is defined to be the smallest
full subcategory of ${\mathcal T}$ that contains ${\mathcal A}$ and is
closed under retracts. For full subcategories ${\mathcal A}$ and 
${\mathcal B}$ of ${\mathcal T}$, let ${\mathcal A}\ast{\mathcal B}$ be
the full subcategory whose objects $L$ occur in a triangle 
$M \to L \to N \to \Sigma M$ with $M \in {\mathcal A}$ and 
$N \in {\mathcal B}$. Then we see that 
$
{\tt thick}_{\mathcal T}^n(C)= {\tt smd}({\tt add}^{\Sigma}(C)^{{\ast}n});
$ 
see \cite{B-B} and \cite[2.2.1]{ABIM}.

A triangular inequality on levels is described in the following lemma.

\begin{lem}{\rm (}cf. \cite[The proof of 6.3.2(3)]{S}{\rm )}\label{lem:thick}
Let ${\mathcal T}$ be a triangulated category and $C$, $C'$ objects in 
${\mathcal T}$. If $\text{\em level}_{\mathcal T}^C M\leq n$ 
and  $\text{\em level}_{\mathcal T}^{C'} C \leq l$, then 
  $\text{\em level}_{\mathcal T}^{C'} M\leq nl$. 
\end{lem}

\begin{proof}
It suffices to prove that if  
$M \in {\tt thick}_{\mathcal T}^n(C)$ and 
 $C \in {\tt thick}_{\mathcal T}^l(C')$, then 
 $M \in {\tt thick}_{\mathcal T}^{nl}(C')$. 

Since the thickening  ${\tt thick}_{\mathcal T}^l(C')$ 
is closed under finite coproducts, all shifts and retracts, 
it follows that 
${\tt add}^{\Sigma}(C) \subset {\tt thick}_{\mathcal T}^l(C')$ and hence 
${\tt thick}_{\mathcal T}^1(C) 
\subset {\tt thick}_{\mathcal T}^l(C')$. 
Assume that ${\tt thick}_{\mathcal T}^i(C) 
\subset {\tt thick}_{\mathcal T}^{il}(C')$ for $i \leq n-1$. 
For any object $M \in {\tt thick}_{\mathcal T}^n(C)$, there exists a
 triangle $M_1 \to M' \to M_2 \to \Sigma M_1$ such that 
$M$ is a retract of $M'$, $M_1 \in {\tt thick}_{\mathcal T}^{n-1}(C)$
 and $M_2 \in {\tt thick}_{\mathcal T}^1(C)$. This yields that 
\begin{eqnarray*}
M &\in & {\tt smd}({\tt thick}_{\mathcal T}^{(n-1)}(C)\ast
{\tt thick}_{\mathcal T}^1(C)) \\
&\subset&  {\tt smd}({\tt smd}({\tt add}^{\Sigma}(C')^{\ast (n-1)l})
\ast {\tt smd}({\tt add}^{\Sigma}(C')^{\ast l})) \\
&=& {\tt smd}
({\tt add}^{\Sigma}(C')^{\ast (n-1)l} \ast 
{\tt add}^{\Sigma}(C')^{\ast l}) \\
&=& {\tt thick}_{\mathcal T}^{nl}(C'). 
\end{eqnarray*}
Observe that the first equality follows from \cite[Lemma 2.2.1]{B-B}.
This completes the proof. 
\end{proof}

\noindent
{\it Proof of Proposition \ref{prop:L}.}
Lemma \ref{lem:thick} and Theorem \ref{thm:main} induce the
inequalities. In fact, we see that 
\begin{eqnarray*}
\text{level}_{\D(C_*(\Omega B))}^{\ C_*(\Omega B)}(C_*(F_f))
&=& \text{level}_{\D(C^*(B))}^{\ \K}(C^*(X)) \\
&\leq&  \text{level}_{\D(C^*(B))}^{\ \K}(C^*(B))\cdot
  \text{level}_{\D(C^*(B))}^{\ C^*(B)}(C^*(X)) \\
&=&
\text{level}_{\D(C_*(\Omega B))}^{\ C_*(\Omega B)}(\K)\cdot
\text{level}_{\D(C^*(B))}^{\ C^*(B)}(C^*(X)) \\
&\leq&\dim H_*(B)\cdot \text{level}_{\D(C^*(B))}^{\ C^*(B)}(C^*(X)) . 
\end{eqnarray*}
We here observe that $C_*(F_{id})\cong \K$ in 
$\D(C_*(\Omega B))$ for the homotopy fibre $F_{id}$ of the identity map 
on $B$. 
The second inequalities follow from the same consideration as above. 
Observe that the based loop space $\Omega B$ is 
the homotopy fibre of the map $* \to B$.  
\hfill\qed

\medskip
\noindent
{\it Proof of Proposition \ref{prop:section}.}
By replacing the square 
\[
\xymatrix@C30pt@R15pt{
B \ar[r]^s \ar[d]_{=} & X \ar[d]^\pi \\
B \ar[r]_{=}  & B, 
}
\]
which is homotopy commutative, to a totally fibred square, we have a
commutative diagram 
\[
\xymatrix@C30pt@R15pt{
 & B \ar[r]^s \ar[d]^\simeq_{\iota_2} & X \ar[d]^{\iota_1}_\simeq \\
\Omega F_\pi \ar[r] & B' \ar[r] & X'
}
\]
in which $\iota_1$ and $\iota_2$ are homotopy equivalences and bottom
sequence is a fibration; see \cite[Propositions 3.2.2 and 3.2.3]{Nei}.  
The map $\iota_1$ gives rise to an equivalence 
$C^*(\iota_1)^* : \D(C^*(X)) \to \D(C^*(X'))$ of triangulated
categories. It is readily seen that $C^*(\iota_1)^*(C^*(B))\cong C^*(B')$. 
This yields that 
\[
\text{level}_{\D(C^*(X'))}^{\ C^*(X')}C^*(B') = \text{level}_{\D(C^*(X))}^{\ C^*(X)}C^*(B).
\]

In view of the Leray-Serre spectral sequence of the path-loop fibration 
$\Omega F_\pi \to PF_\pi \to F_\pi$, we see that $H^*(F_\pi; \K)=\K$ 
if and only if 
$H^*(\Omega F_\pi; \K)$ and $H^*(F_\pi; \K)$ are of finite dimension. 
Observe that $F_\pi$ is simply-connected since $\pi$ has a right
inverse. 
By Corollary \ref{cor:finite} (2), we have the result.   
\hfill\qed 

\medskip
Before proving Proposition \ref{prop:level_2^n}, we consider a special case
for the assertion.    

\begin{prop}\label{prop:level_2}Let $F \to X \to B$ be a fibration with 
$B$ simply-connected. Suppose that $H^*(F; \K)\cong H^*(S^l; \K)$ 
as a graded vector space. Then one has 
\[
\text{\em level}_{\DD(C^*(B; \K))}^{\ C_*(B; \K)}C^*(X; \K) \leq 2.
\] 
Moreover, 
$\text{\em level}_{\DD(C^*(B; \K))}^{\ C_*(B; \K)}C^*(X; \K) = 2$ 
if and only if $H^*(X; \K)$ is not a free $H^*(B; \K)$-module. 
\end{prop}

The following lemma serves to prove Proposition \ref{prop:level_2}. 

\begin{lem}\label{lem:level_1} 
Let $X \to B$ be an object in ${\mathcal TOP}_B$. Then 
$\text{\em level}_{\DD(C^*(B))}^{\ C_*(B)}C^*(X)=1$ if and only if 
$H^*(X; \K)$ is a free $H^*(B; \K)$-module. 
\end{lem}

\begin{proof}
Suppose that $\text{level}_{\D(C^*(B))}^{\ C_*(B)}C^*(X)=1$. 
Then by definition, we see that $C^*(X)$ is a retract of a free
 $C^*(B)$-module. Therefore $H^*(X)$ is a projective $H^*(B)$-module 
and hence $H^*(X)$ is a free $H^*(B)$-module; see 
\cite[page 274 Remark 1]{F-H-T} for example. 

We see that 
$\text{level}_{\D(H^*(B))}^{\ H_*(B)}H^*(X)=1$
 if $H^*(X)$ is a free $H^*(B)$-module. 
The result \cite[Corollary 7.3]{K5} implies that 
$\text{level}_{\D(C^*(B))}^{\ C_*(B)}C^*(X)$ is less than or equal to  
$\text{level}_{\D(H^*(B))}^{\ H_*(B)}H^*(X)$. 
This completes the proof.   
\end{proof}

\noindent
{\it Proof of Proposition \ref{prop:level_2}.} 
Let $S_l$ be the homotopy fibre of the projection $X \to B$. We observe
that $H^*(S_l; \K)\cong H^*(F; \K)\cong H^*(S^l; \K)$. 
In order to prove the proposition, 
it suffices to show that 
$\text{level}_{\D(C_*(\Omega B))}^{\ \K}C_*(S_l)\leq 2$.  This follows from Theorem \ref{thm:main} (2). 
We define a DG subalgebra $R$ of $C_*(\Omega B)$ by 
$R_0=\K$, $R_1=\text{Ker} \ d$ and $R_{\geq 2}= C_{\geq 2}(\Omega B)$. 
It is immediate that the inclusion $i : R \to C_*(\Omega B)$ is a
quasi-isomorphism. Then the map $i$ induces 
an equivalence of categories 
$i^* : \D(C_*(\Omega B)) \to \D(R)$. 
Moreover, we have $i^*(\K) = \K$ and $i^*(C_*(S_l))=C_*(S_l)$. 
Therefore, we conclude that 
\[
\text{level}_{\D(C_*(\Omega B))}^{\ \K}C_*(S_l) = 
\text{level}_{\D(C_*(R))}^{\ \K}C_*(S_l).
\]

Let $N$ be a DG $R$-submodule of $C_*(S_l)$ defined by 
$N_{\leq l-1}=0$, $N_l=\text{Im} \ d$ and $N_i =C_i(S_l)$ for $i> l$. 
Since $N$ is acyclic, it follows that the projection  
$C_*(S_l) \to C_*(S_l)/N$ is a quasi-isomorphism 
of $R$-modules.  
Moreover, we can construct a triangle in $\D(R)$ of the form 
$
\Sigma^{-l}\K \to C_*(S_l)/N \to \K \to .
$ 
In fact, the projection from 
the quotient $(C_*(S_l)/N)\bigl/\bigr. \Sigma^{-l}\K$ to $\K$ 
is a quasi-isomorphism of $R$-modules. 
Then the triangle yields that 
$
\text{level}_{\D(R)}^{\ \K}C_*(S_l) = 
\text{level}_{\D(R)}^{\ \K}C_*(S_l)/N \leq 2.
$ We have the result.  
The latter half of the assertion follows from Lemma \ref{lem:level_1}.  
\hfill \qed

\begin{rem}
We can prove Proposition \ref{prop:level_2} by means of a minimal 
semifree resolution $\Gamma \stackrel{\simeq}{\to} C^*(X)$ of 
$C^*(X)$ as a $C^*(B)$-module. Indeed, we see that 
\[
H^*(\K\otimes_{C^*(B)}\Gamma)=H^*(\K\otimes^L_{C^*(B)}C^*(X))
= H^*(S^l) = \K\{1, w\}, 
\]
where 
$\deg 1 = 0$ and $\deg w =l$.  This implies that the filtration of $F$ 
has class at most $2$; see \cite[4.1]{ABIM}. 
Proposition \ref{prop:level_2} follows from 
\cite[Theorem 4.1]{ABIM}. 
\end{rem}

\medskip
We use again Lemma \ref{lem:thick} to prove Proposition \ref{prop:level_2^n}. 

\begin{proof}[Proof of Proposition \ref{prop:level_2^n}]
Consider the maps 
$C^*(B) \stackrel{\alpha}{\to} C^*(X_i) \stackrel{p_{i+1}^*}{\to}
C^*(X_{i+1})$, where $\alpha = (p_i\circ \cdots \circ p_1)^*$. 
Then the map $\alpha$ induces an exact functor 
$\alpha^* : \D(C^*(X_i)) \to \D(C^*(B))$. 
In view of Proposition \ref{prop:level_2} 
and Lemma \ref{lem:2}, we have 
\[
2 \geq \text{level}_{\D(C^*(X_i))}^{\ C^*(X_i)}C^*(X_{i+1})
\geq \text{level}_{\D(C^*(B))}^{\ \alpha^*C^*(X_i)}\alpha^*C^*(X_{i+1})
= \text{level}_{\D(C^*(B))}^{\ \alpha^*C^*(X_i)}C^*(X_{i+1}). 
\]
Therefore, Lemma \ref{lem:thick} and the induction hypothesis 
allow us to deduce that
\[
 \text{level}_{\D(C^*(B))}^{\ C^*(B)}C^*(X_{i+1})
\leq 
\text{level}_{\D(C^*(B))}^{\ C^*(B)}\alpha^*C^*(X_i)\cdot
\text{level}_{\D(C^*(B))}^{\ \alpha^*C^*(X_i)}C^*(X_{i+1})
\leq 2^i\cdot 2. 
\]
We have the first inequality. 
The second inequality follows from Proposition \ref{prop:L}. 
\end{proof}

\section{Proofs of Theorem \ref{thm:ghost} and Proposition \ref{prop:shriek}}

We begin by recalling a shriek map on the classifying space $BG$ 
of a connected Lie group $G$. The classifying space $BG$ is a Gorenstein space of dimension $-\dim G$; 
see \cite{FHT_G} for more details. 
Then the result \cite[Theorem 12]{F-T} deduces that 
$\text{Ext}_{C^*(BG^{\times n})}^*(C^*(BG), C^*(BG^{\times n}))\cong H^{*-(n-1)(-\dim G)}(BG)$, where 
the DG right $C^*(BG^{\times n})$-module structure on $C^*(BG)$ is induced by the diagonal map 
$\Delta^{(n-1)} : BG \to BG^{\times n}$. In particular, we have a generator of the vector space 
$$\text{Ext}_{C^*(BG^{\times n})}^{-(n-1)\dim G}(C^*(BG), C^*(BG^{\times n}))\cong H^0(BG)=\K, $$ 
which is called a {\it shriek map} associated with the diagonal map.

\begin{proof}[Proof of Theorem \ref{thm:ghost}] 
By assumption, the cohomology $H^*(BG)$ is a polynomial algebra, say $H^*(BG) = \K[x_1, ..., x_s]$. 
Then $H^*(G)$ is isomorphic to the algebra with a $2$-simple system of generators $s^{-1}x_1, ..., s^{-1}x_s$,  
where $\deg s^{-1}x_i = \deg x_i -1$; see \cite[page 154]{McC}. 
Observe that $H^*(G)$ is the exterior algebra generated by $s^{-1}x_1, ..., s^{-1}x_s$ if the characteristic of $\K$ is odd.  

\begin{claim} \label{claim:1} 
In the Leray-Serre spectral sequence $\{{}_{LS}E^{*,*}_r, d_r\}$ of the fibration  
$G \to BG^{\times (k-1)} \stackrel{B(1\times \Delta)}{\to} BG^{\times k}$, 
the generators $s^{-1}x_i$ are transgressive. More precisely, for the transgression $\tau$, one has 
$\tau(s^{-1}x_i) = \lambda_i(x_i\otimes 1 - 1\otimes x_i)$ for some non-zero scalar $\lambda_i$ under an isomorphism 
$H^*(BG^{\times k})\cong H^*(BG^{\times (k-2)})\otimes H^*(BG) \otimes H^*(BG)$. 
\end{claim}

Therefore, there is no non-trivial element in ${}_{LS}E_\infty^{0, *}$ for $* > 0$. 
This implies that  that the shriek map $B(1\times \Delta)^! : C^*(BG^{\times (k-1)}) \to C^{*-d}(BG^{\times k})$ is a ghost map, where 
$d=\dim G$.  
In fact, the induced map $H^*(B(1\times \Delta)^!)$ is the integration along the fibre; see \cite[Theorems 5 and 13]{F-T}. 

We shall prove that the composition of the shriek maps $B(1\times \Delta)^! \circ \cdots \circ B\Delta^! : 
C^*(BG) \to C^{*- (n-1)d}(BG^{\times n})$ 
is non-trivial in $\D(C^*(BG^{\times n}))$. 
To this end, we consider the homotopy pullback square 
\[
\xymatrix@C25pt@R15pt{
G^{n-1} \ar[r] \ar@{=}[d]& 
LBG\times_{BG} \cdots \times_{BG} LBG \ar[r]^(0.7){\widetilde{\Delta}} \ar[d]_{ev_{0}} & LBG 
\ar[d]^{ev_{a_1, ..., a_n}}  \\
G^{n-1} \ar[r]  &BG \ar[r]_{\Delta^{(n-1)}} & BG^{\times n},
}
\eqnlabel{add-4}
\]
where  $ev_{a_1, ..., a_n}$ denotes 
the evaluation map at points $a_k= \frac{k-1}{n}$ for $k = 1, ...., n$.   
We regard the composite $B(1\times \Delta)^! \circ \cdots \circ B\Delta^!$ as the shriek map  
$(\Delta^{(n-1)})^!$ by choosing an appropriate orientation class of the fibration $\Delta^{(n-1)}$; 
see \cite[Section 2.3, Composition]{C-M} for example. 
In order to show non-triviality of the shriek map $(\Delta^{(n-1)})^!$, it suffices to prove that the shriek map 
$\widetilde{\Delta}^! : C^*(LBG\times_{BG} \cdots \times_{BG} LBG) \to C^{*-(n-1)d}(LBG)$ is non-trivial since 
$\widetilde{\Delta}^!$ is an extension of  $(\Delta^{(n-1)})^!$; see the proof of \cite[Theorem 6]{F-T}. 
We observe that 
$
H^*(LBG)\cong H^*(BG)\otimes \Delta(s^{-1}x_1, ...., s^{-1}x_s) 
$
as an algebra. 

Let ${\mathcal K} = H^*(BG) \otimes \wedge(s^{-1}QH^*(BG))\otimes H^*(BG) \to H^*(BG)$ be 
the two sided bar resolution of $H^*(BG)$; see \cite{B-S} for example. Then we have 
a projective resolution 
\[
{\mathcal K}^{\otimes_{H^*(BG)}^{n-1}} \to H^*(BG)^{\otimes_{H^*(BG)}^{n-1}}=H^*(BG)
\]
of $H^*(BG)$ 
as an $H^*(BG)^{\otimes n-1}$-module. 

Let $\{E_r, d_r\}$ be the Eilenberg-Moore spectral sequence for the right-hand side pullback in the diagram (5.1). 
Computing  the $E_2$-term by using the projective resolution 
${\mathcal K}^{\otimes_{H^*(BG)}^{n-1}} \to H^*(BG)$
mentioned above,  we see that 
\begin{eqnarray*}
E^{*,*}_2&\cong &H\big(
H^*(BG)^{\otimes n}\otimes \Delta(s^{-1}x_{1,1}, ...., s^{-1}x_{1,s}, ...., 
s^{-1}x_{n-1, 1}, ...., s^{-1}x_{n-1, s}) \\
& & \  \  \otimes_{H^*(BG)^{\otimes n}}H^*(LBG),  \  
\partial(s^{-1}x_{i,j}) = ev^*_{a_1, ...a_n}(x_{i, j}\otimes 1 - 1\otimes x_{i+1, j})\big)
\end{eqnarray*}
as a bigraded algebra. Since $ev_0 \simeq ev_{a_k}$ for any $k$, it follows that 
$ev^*_{a_1, ...a_n}\circ p_k^* \simeq ev_0$, 
where $p_k : BG^{\times n}$ denotes the projection into the $k$th factor.  
This implies that  $ev^*_{a_1, ...a_n}(x_{i, j}\otimes 1 - 1\otimes x_{i+1, j})= 0$ 
since $ev_0^*(x_i) = x_i$. For dimensional reasons, we see that   
$
E_\infty^{*,*} \cong H^*(LBG)\otimes \Delta(s^{-1}x_{1,1}, ...., s^{-1}x_{1,s}, ...., 
s^{-1}x_{n-1, 1}, ...., s^{-1}x_{n-1, s}). 
$
This fact enables us to conclude that the Leray-Serre spectral sequence of the upper fibration in the homotopy 
pull-back above collapses at the $E_2$-term. Therefore, it follows that the integration along the fibre 
$H^*((\widetilde{\Delta})^!)$ is non-trivial. This completes the proof.   
\end{proof}

\begin{proof}[Proof of Claim \ref{claim:1}] 
We consider a morphism of homotopy fibrations 
\[
\xymatrix@C25pt@R15pt{
G \ar[d] \ar@{=}[r] & G \ar[d] \\ 
BG^{\times (k-2)}\times BG \ar[d]_{1\times \Delta_{BG}} & BG^{\times (k-1)} \ar[l]_(0.4){\simeq}  \ar[d]^{B(1\times \Delta)}\\
BG^{\times (k-2)}\times BG\times BG & BG^{\times k} \ar[l]^(0.3){\simeq} 
}
\]
in which horizontal maps are homotopy equivalences. Thus in order to prove Claim \ref{claim:1}, 
it suffices to show that the result holds 
for the spectral sequence of the fibration $G \to BG \stackrel{\Delta_{BG}}{\to} BG\times BG$.  

Let $z_i : BG \to K:=K(\K, \deg x_i)$ be the map corresponding to the generator $x_i$ of $H^*(BG)$; that is, $z_i^*(\iota) = x_i$ 
for the fundamental class $\iota$ of $K$.   
In the Leray-Serre spectral sequence of the homotopy fibration 
$K(\K, \deg x-1) \to K \stackrel{\Delta_K}{\to} K \times K$, 
the transgression sends the fundamental class of the fibre to the element $\iota\otimes 1 - 1\otimes \iota$ 
up to the multiplication by a non-zero scalar 
because $\Delta_K^*(\iota\otimes 1 - 1\otimes \iota) = 0$. 
The naturality of the morphism induced by $z_i$ implies that  $\tau(s^{-1}x_i) = \lambda_i(x_i\otimes 1 - 1\otimes x_i)$ for some non-zero scalar 
$\lambda_i$. We have the result.  
\end{proof}

\begin{rem}\label{rem:generator} 
The proof of Theorem \ref{thm:ghost} enables us to conclude 
that the shriek map $(\Delta^{(n-1)})^!$ is the non-trivial generator in 
$\text{Ext}_{C^*(BG^{\times n})}^{-(n-1)\dim G}(C^*(BG), C^*(BG^{\times n}))$ and it is a ghost map. 
\end{rem}

\begin{proof}[Proof of Proposition \ref{prop:shriek}] We have a fibration of the form 
$G^{\times(n-1)} \to BG \stackrel{\Delta^{(n)}}{\to} BG^{\times n}$. Therefore, we see that 
$H^*(G^{\times (n-1)}) \cong \text{Tor}_{*}^{H^*(BG^{\times n})}(H^*(BG), \K)$ and hence the torsion product is of finite dimension. 
Then it follows from \cite[Lemma 7.1]{K5} that 
$\text{level}_{\D(C^*(BG^{\times n})}^{C^*(BG)^{\times n}}(C^*(BG)) \leq 
\text{pd}_{H^*(BG^{\times n})}H^*(BG)+1$, where $\text{pd}_{A}M$ denotes the projective dimension of an $A$-module $M$.   
Let 
\[
{\mathcal K}^{\otimes_{H^*(BG)}^{n-1}} \to H^*(BG)^{\otimes_{H^*(BG)}^{n-1}}=H^*(BG)
\]
be the projective resolution of $H^*(BG)$ 
as an $H^*(BG)^{\otimes n-1}$-module introduced in the proof of Theorem \ref{thm:ghost},
This yields that 
\[
\text{pd}_{H^*(BG^{\times n})}H^*(BG) \leq (n-1)\dim QH^*(BG).
\]
We have the upper bound of the level. 
Proposition \ref{prop:ghost} and Theorem \ref{thm:ghost} give the lower bound.  

We prove the latter half of the assertion. Since $H^*(BG; \K)$ is a polynomial algebra generated by elements with even degree, 
it follows from \cite[Proposition 2.4]{K4} that the homotopy fibration $G^{\times (n-1)} \to BG \stackrel{\Delta^{(n)}}{\to} BG^{\times n}$ is 
$\K$-formalizable; see \cite[Section 2]{K5}.
Thus the result \cite[Proposition 5.2]{K5} implies that 
\[
\text{level}_{\D(C^*(BG^{\times n})}^{C^*(BG^{\times n})}(C^*(BG)) =
 (n-1) \dim QH^*(BG; \K)+1. 
\]
This completes the proof.
\end{proof}

\medskip
\noindent
{\it Acknowledgments}. I thank Jim Stasheff for his interest in this work and for 
comments on Theorem \ref{thm:formal_DGA}. I benefited from  
inspiring discussions with Luc Menichi about string operations on 
the classifying spaces without which Theorem \ref{thm:ghost} could not have been obtained.   
I am grateful to Younggi Choi for a conversation on 
the proof of Claim \ref{claim:1}. I also thank the referee for valuable suggestions to revise a previous version of this paper.

\section{Appendix: A variant of Koszul duality for DG algebras}

In this section, we describe a result concerning Theorem \ref{thm:diagram}, which is regarded as a variant of the Koszul 
duality for DG algebras.  
We also refer the reader to the paper \cite{P} due to Positselski for a more general approach to the derived Koszul duality. 

We begin by recalling the result on a coderived category 
due to Lef\'evre-Hasegawa \cite{L-H}. 

Let $(A, d_A, \e_A)$ and $(C, d_C, \e_C)$ be 
an augmented DG algebra and a co-augmented DG coalgebra over 
a field $\K$, respectively. 
By using the kernel $\overline{C}$ of the counit of $C$, we have a decomposition $C=\overline{C}\otimes \K$. 
Let  $\overline{\Delta} : \overline{C} \to \overline{C}\otimes \overline{C}$ denote the reduced coproduct defined by 
$\overline{\Delta}(x) = \Delta(x) - x\otimes 1 - 1\otimes x$. We say that 
a coaugmented DG coalgebra is {\it cocomplete} if 
$\overline{C} = \cup_{l\geq 1}\text{Ker} \ (\overline{\Delta}^{(l)} : \overline{C} \to \overline{C}^{\otimes l+1})$, where 
$\overline{\Delta}^{(l)}$ is the iterated coproduct defined by 
$\overline{\Delta}^{(l)} =(\overline{\Delta}\otimes 1^{\otimes l-1})\circ \cdots \circ (\overline{\Delta}\otimes 1)
\circ \overline{\Delta}$. 
By definition, a {\it twisted cochain} $\tau : C \to A$ is 
a $\K$-linear map of degree $+1$ such that $\e_A\circ \tau\circ\e_C= 0$
and 
\[
d_A\circ \tau + \tau \circ d_C + \mu_A\circ (\tau\otimes \tau)\circ
\Delta_C
= 0,
\]
where $\mu_A$ and $\Delta_C$ are the multiplication of $A$ and the
comultiplication of $C$, respectively. Let $M$ be a DG right $A$-module. 
Then we defined the {\it twisted tensor product} $M\otimes_\tau C$ to
be the comodule $M\otimes C$ over $C$ 
endowed with the differential 
\[
d = d_M\otimes 1 + 1\otimes d_C - (\mu_M\otimes 1)(1\otimes \tau\otimes
1)(1\otimes \Delta_C). 
\]
For a DG $C$-comodule $N$, we define the DG module 
$N\otimes_\tau A$ similarly. 
Let $\Delta_N$ be the comodule structure of a DG $C$-comodule $N$. 
We say that $N$ is {\it cocomplete} if 
$
N = \cup_{l\geq 1}\text{Ker} \ (\overline{\Delta}_N^{(l)} : N \to N\otimes\overline{C}^{\otimes l})$, where 
$\overline{\Delta}_N(x) = \Delta_N(x)-x\otimes 1$ for $x\in N$ and $\overline{\Delta}_N^{(l)}$ denotes the iterated 
comodule structure defined by the same way as the iterated coproduct on $\overline{C}$. 

Let $C$ be a cocomplete DG coalgebra and 
$\tau_0 : C \to \Omega C$ the canonical twisting cochain. 
Then the category $\mathsf{comod-}C$ of
cocomplete DG comodules over $C$ admits the structure of a model
category for which $f : N \to N'$ is a weak equivalence, by definition, 
if and only if 
$
f\otimes 1 : N\otimes_{\tau_0}\Omega C \to N'\otimes_{\tau_0}\Omega C
$
is a quasi-isomorphism. 
For the details, see \cite[Th\'eor\`eme 2.2.2.2]{L-H}. 
Observe that $f$ is a weak equivalence, then $f$ is a quasi-isomorphism. 
This fact follows form \cite[Proposition 2.14]{FHT}. 
We define the coderived category $\D_c(C)$, which is a triangulated
category, to be the localization of the homotopy category of 
$\mathsf{comod-}C$ with respect to the class of all weak equivalences. 

\begin{rem} \label{rem:coalgebra}
Let  $C$ be a finite dimensional co-augmented coalgebra. 
The result \cite[1.6.4]{Mo} due to Montgomery allows one 
to deduce that the functor 
$F : \mathsf{comod-}C \to C^\vee\mathsf{-mod}$ mentioned in Section 3 is an equivalence of categories. 
As mentioned above, weak equivalences between cocomplete DG-comodules are 
quasi-isomorphisms. Then we see that $F$ induces a functor $F_* : \D( \mathsf{comod-}C) \to D(C^\vee\mathsf{-mod})$ of 
triangulated categories. 
Observe that the functor $F_*$ is {\it not} an equivalence of triangulated categories in general. In fact, we can regard 
the exterior algebra $\wedge (x)$ as a Hopf algebra with a primitive element $x$ of degree $-1$.  Forgetting the algebra structure 
of $\wedge (x)$, we have a DG coalgebra $C_1$ endowed with the trivial differential. 
The argument in \cite[Section 4]{Keller2} asserts that 
in $(C_1)^\vee\mathsf{-mod}$, weak equivalences form a strictly smaller class than that of quasi-isomorphisms. 

On the other hand, the equivalence $F$ allows us to obtain an equivalence 
\[
\widetilde{F_*} : \D_c(C)=\D(\mathsf{comod-}C) \stackrel{\simeq}{\longrightarrow} 
\widetilde{\D}(C^\vee\mathsf{-mod})
\]
of triangulated categories. Here  $\widetilde{\D}(C^\vee\mathsf{-mod})$ denotes the localization 
of the homotopy category of $C^\vee\mathsf{-mod}$ with respect to the class of morphisms 
which come from weak equivalences in $\mathsf{comod-}C$ by $F$.
\end{rem}

The following theorems assert that a coderived category is closely related to a derived category. 

\begin{thm}\label{thm:L-R}\cite[2.2.3, Lemma 2.2.1.2, Proposition 2.2.4.1]{L-H} 
Let $\tau : C \to A$ be a
 twisting cochain. Then one has adjoint functors
\[
\xymatrix@C30pt@R15pt{
\DD_c(C) \ar@<1ex>[rr]^{L:= - \otimes_{\tau}A} & &  
\DD(A)   \ar@<1ex>[ll]^{R:= - \otimes_{\tau}C} 
}
\]
between triangulated categories. 
\end{thm}

\begin{thm}\label{thm:ev}\cite[Proposition 2.2.4.1]{L-H}
The following are equivalent. \\
{\rm (i)} The map $\tau$ induces a quasi-isomorphism $\Omega (C) \to A$. \\
{\rm (ii)} The canonical map $A\otimes_\tau C\otimes_\tau A \to A$ is 
a quasi-isomorphism. \\
{\rm (iii)} The functor $L$ and $R$ in Theorem \ref{thm:L-R} 
are equivalences.  
\end{thm}

Let $V$ be a finite dimensional, non-negatively 
graded vector space with $V^{\text{odd}}=0$. 
Let $SV$ be the polynomial algebra and $\wedge \Sigma V$ 
the primitively generated coalgebra whose underlying space is the exterior algebra 
on $\Sigma V$. 
Then the projection from 
$\wedge \Sigma V$ to  $\Sigma V$ and the inclusion form $V$ to $SV$ 
give rise to a twisting cochain 
$\tau : \wedge \Sigma V \to SV$. 
Thus we have exact functors between derived and coderived categories  
\[
\xymatrix@C30pt@R15pt{
\widetilde{\D}((\wedge \Sigma V)^\vee\mathsf{-mod})  & 
\D_c(\wedge \Sigma V) \ar@<1ex>[rr]^{L:= - \otimes_{\tau}SV} 
\ar[l]_(0.37){\widetilde{F_*}} & & 
\D(SV)   \ar@<1ex>[ll]^{R:= - \otimes_{\tau}\wedge \Sigma V}. 
}
\]
Here $\widetilde{F_*}$ stands for the functor defined in Remark \ref{rem:coalgebra}.

The existence of the two-sided Koszul resolution (see for example \cite{B-S})  
implies that the functor $R$ gives an equivalence with inverse $L$. 
Indeed this follows from the equivalence of 
the assertions (ii) and (iii) in Theorem \ref{thm:ev}.  
Moreover, since the vector space $V$ is of finite dimension, 
the functor $\widetilde{F_*}$ is also an equivalence between 
$\D( \wedge \Sigma V)$ and $\D((\wedge \Sigma V)^\vee \mathsf{-mod})$; 
see Remark \ref{rem:coalgebra}.   

More generally, the proof of \cite[Theorem 4.4]{H-W} due to He and Wu enables us to deduce
the following result. 

\begin{thm} {\em (}cf. \cite[Theorem 7.4]{ABIM}, \cite[Theorem 4.7]{H-W} {\em )} \label{thm:formal_DGA} 
Let $A$ be a locally finite, simply-connected 
DG algebra over a field $\K$. 
Suppose that the dual $({B}A)^\vee$ to the bar construction 
is formal in the sense that  $({B}A)^\vee$ admits a $TV$-model 
$TV \stackrel{\simeq}{\to} ({B}A)^\vee$ together with 
a quasi-isomorphism $TV \stackrel{\simeq}{\to} 
H(({B}A)^\vee)=\text{\em Ext}_A(\K, \K)$. 
Assume further that $\text{\em Ext}_A(\K, \K)$ is of finite dimension. Then one has equivalences 
\[
\xymatrix@C30pt@R15pt{
\widetilde{\DD}(\text{\em Ext}_A(\K, \K) \mathsf{-mod})\ar@<1ex>[r]^(0.7)h_(0.7){\simeq} & 
\DD(A)  \ar@<1ex>[l]^(0.3)t
}
\] 
of triangulated categories. If $A$ is $2$-connected, then $t$ satisfies the condition that 
$t(\K)= \text{\em Ext}_A(\K, \K)^\vee$ and $t(A)=\K$ in 
$\widetilde{\DD}(\text{\em Ext}_A(\K, \K)  \mathsf{-mod})$. 
\end{thm}

Let $A$ be a $2$-connected DG algebra as in Theorem \ref{thm:formal_DGA}. Then 
it follows that  
for an object $M$ in $\D(A)$,   
\[
\text{level}_{\D(A)}^{\ A}(M) = \text{level}_
{\widetilde{\D}(\text{Ext}_A(\K, \K)\mathsf{-mod})}^{\ \K}(t(M)).
\]

\begin{proof}[Proof of Theorem \ref{thm:formal_DGA}.]
Since $A$ is simply-connected and locally finite, 
it follows that the bar construction is
 also locally finite. Thus we can assume that for the $TV$-model 
$TV \stackrel{\simeq}{\to} ({B}A)^\vee$, the
 graded vector space $V$ is locally finite; 
 see the proof of Proposition \ref{prop:equivalence}.  Then the sequence of 
quasi-isomorphisms 
\[
E:= \text{Ext}_A(\K, \K)=H(({B}A)^\vee) 
\stackrel{\simeq}{\longleftarrow} TV \stackrel{\simeq}{\longrightarrow} 
({B}A)^\vee
\]
of DG algebras gives rise to a sequence of quasi-isomorphisms
\[
\Omega(E^\vee) \stackrel{\simeq}{\longrightarrow} \Omega (TV^\vee)  
\stackrel{\simeq}{\longleftarrow}  
\Omega(({B}A)^{\vee \vee}) \stackrel{\simeq}{\longleftarrow}
\Omega({B}A) \stackrel{\simeq}{\longrightarrow} A
\]
as DG algebras. Thus we have equivalences 
\[
\xymatrix@C30pt@R15pt{
\D(\Omega (E^\vee))\ar@<1ex>[r]^(0.6)\alpha_(0.6){\simeq} & 
\D(A)  \ar@<1ex>[l]^(0.4)\beta
}
\]
of triangulated categories for which $\beta(A)=\Omega(E^\vee)$ and 
$\beta(\K)= \K$. 
The canonical twisting cochain 
$\tau_0 : E^\vee \to \Omega (E^\vee)$ induces 
the identity map $\Omega(E^\vee) \to \Omega(E^\vee)$. 
In view of Theorem \ref{thm:ev}, we have equivalences
\[
\xymatrix@C30pt@R15pt{
\D({\mathsf{comod-}}(E^\vee))\ar@<1ex>[rr]^(0.6)L_(0.6){\simeq} & &
\D(\Omega (E^\vee)).   
\ar@<1ex>[ll]^(0.4){R=-\otimes_{\Omega (E^\vee)}{B}(\Omega(E^\vee); \Omega (E^\vee))}
}
\]
Since $E^\vee$ is a finite dimensional coalgebra by assumption, it
 follows that the functor  
\[
\widetilde{F_*} : \D({\mathsf{comod-}}(E^\vee)) 
\stackrel{\simeq}{\longrightarrow} \widetilde{\D}(E\mathsf{-mod}),
\]
which is defined in Remark \ref{rem:coalgebra} 
gives an equivalence of triangulated categories. 
Then one has an equivalence 
$t:= \widetilde{F_*} \circ R\circ \beta : \D(A) \to \widetilde{\D}(E\mathsf{-mod})$. 

The natural map $\sigma : E^\vee \to B\Omega(E^\vee)$ and 
$\eta : B(\Omega(E^\vee); \Omega(E^\vee)) \to \K$ are quasi-isomorphisms; see \cite[Propositions 2.4 and 2.14]{FHT}.  
If $A$ is $2$-connected, then $E^\vee$ is simply-connected. Then it follows from 
\cite[Remark 2.3]{FHT} that maps $\sigma$ and $\eta$ are weak equivalences. This implies the latter half of the theorem.  
\end{proof}

The following proposition provides 
examples of DG algebras which satisfy the assumptions in 
Theorem \ref{thm:formal_DGA}.  

\begin{prop}
\label{prop:ex}
Let $E$ be a non-positively graded, connected DG
 algebra; that is, $E^0 = \K$ and  $E^i=0$ for $i> 1$. 
Suppose further that $E$ is formal and of finite dimension. Put
 $A=\Omega(E^\vee)$. Then the algebra $({B}A)^\vee$ is a formal and 
$H(({B}A)^\vee)\cong \text{\em Ext}_A(\K, \K) \cong H(E)$ 
as algebras. In consequence, 
the DG algebra $A$ satisfies all the assumptions in 
Theorem \ref{thm:formal_DGA}. Thus one has equivalences 
\[
\xymatrix@C30pt@R15pt{
\widetilde{\DD}(H(E)\mathsf{-mod})\ar@<1ex>[r]^(0.55)h_(0.55){\simeq} & 
\DD(\Omega(E^\vee)) \ar@<1ex>[l]^(0.45)t
}
\]
of triangulated categories. Assume further that $E^\vee$ is simply-connected. Then one has $t(\Omega(E^\vee))=\K$ and
 $t(\K)=H(E)^\vee$.  
\end{prop}

\begin{proof}
Since $E$ is a finite dimensional DG algebra, it follows from 
\cite[Proposition 2.14]{FHT} that there exists 
a quasi-isomorphism 
$\alpha : E^\vee \stackrel{\simeq}{\longrightarrow} 
{B}A = {B}\Omega (E^\vee)$
of coalgebras. Let $\eta : TV \stackrel{\simeq}{\longrightarrow} 
({B}\Omega (E^\vee))^\vee$ be a TV-model. 
We then have a sequence 
\[
E \cong E^{\vee \vee} \mapleftud{\simeq}{\alpha^\vee \circ \eta} \ 
TV \maprightud{\simeq}{\eta} ({B}A)^\vee 
\]
of quasi-isomorphism of DG algebras. 
By assumption, the DG algebra $E$ is formal. This enables us to obtain
 quasi-isomorphisms
$H(E) \stackrel{\simeq}{\longleftarrow} TW \stackrel{\simeq}{\longrightarrow}  E$.  
The lifting lemma \cite[Lemma 3.6]{FHT2} yields a quasi-isomorphism 
$TV \stackrel{\simeq}{\longrightarrow}  H(E)$  
of DG algebras and hence $({B}A)^\vee$ is
 formal. 
\end{proof}

\begin{ex} \label{ex:E-P} 
Let $E$ be an exterior algebra $\wedge (x_1, ..., x_n)$ 
generated by $x_1$, ..., $x_n$, where   
$-\deg x_i$ is odd for any $i$. We have an isomorphism 
$H(\Omega (E^\vee)) 
\cong H(({B}E)^\vee)=\text{Tor}_E(\K, \K)^\vee$ 
of algebras which sends the cycles $\langle x_i^\vee \rangle$ to $[x_i]^\vee$. 
Moreover, there exists an isomorphism 
\[
\eta : \text{Tor}_E(\K, \K)=H({B}E) \stackrel{\cong}{\to} \Gamma[sx_1 , ..., sx_n]
\]
of coalgebras such that $\eta([x_i])=sx_i$, where 
$\deg sx_i = \deg x_i - 1$ and 
$\Gamma[sx_1 , ..., sx_n]$ stands for the divided power Hopf algebra with the
 comultiplication
 $\Delta(\gamma_i(sx_j))=\sum_{k+l=j}\gamma_i(sx_j)\otimes
 \gamma_k(sx_j)$; see the proof of \cite[Lemma 1.5]{K1}. 
Thus we see that the algebra 
$H(\Omega (E^\vee))$ is isomorphic to the polynomial algebra 
$\K[sx_1^\vee, ..., sx_n^\vee]$, 
where $\deg sx_i^\vee = -\deg x_i + 1$. 
Since the algebra $\Omega (E^\vee)$ is free, it follows that 
there exists a quasi-isomorphism 
$
\theta :  \Omega (E^\vee) \stackrel{\simeq}{\to} \K[sx_1^\vee , ..., sx_n^\vee]
$
of algebras such that $\theta(\langle x_i^\vee \rangle)=sx_i^\vee$ for
 $i$. This implies that $\Omega (E^\vee)$ is  formal. 
Therefore,   
Theorem \ref{thm:formal_DGA} and Proposition \ref{prop:ex} enable us to obtain 
equivalences 
\[
\xymatrix@C30pt@R15pt{
\widetilde{\D}(\wedge (x_1, ..., x_n)\mathsf{-mod})\ar@<1ex>[r]^(0.55)h_(0.55){\simeq} & 
\D(\K[sx_1^\vee , ..., sx_n^\vee]) \ar@<1ex>[l]^(0.45)t
}
\]
of triangulated categories. 
This result is a variant of \cite[Theorem 7.4]{ABIM}; see also 
\cite[Section 4]{Keller2}. 
\end{ex}

\end{document}